\documentclass[a4paper,11pt]{article}

\usepackage[left=2.2cm, right=2.2cm, top=2.2cm, bottom=2.2cm]{geometry}

\usepackage{graphicx}%
\usepackage{enumitem}
\usepackage{amsmath,amssymb,amsfonts}%
\usepackage{amsthm}%
\usepackage{hyperref}

\theoremstyle{plain}
\newtheorem{theorem}{Theorem}[section]
\newtheorem{proposition}[theorem]{Proposition}
\newtheorem{lemma}[theorem]{Lemma}
\newtheorem{corollary}[theorem]{Corollary}

\theoremstyle{definition}
\newtheorem{definition}[theorem]{Definition}
\newtheorem{example}[theorem]{Example}
\newtheorem{remark}[theorem]{Remark}

\numberwithin{equation}{section}

\def\L{\mathsf{L}}
\def\U{\mathsf{U}}

\def\LL{\mathcal{L}}
\def\UU{\mathcal{U}}
\def\SS{\mathcal{S}}
\def\TT{\mathcal{T}}

\setlist[itemize]{label=$\circ$, itemsep=1pt, topsep=2pt}
\setlist[enumerate]{itemsep=-5pt, topsep=2pt}

\title{Relational correspondences for $L$-fuzzy rough approximations defined on De~Morgan Heyting algebras}

\author{Jouni J{\"a}rvinen%
\thanks{These authors contributed equally to this work}\,
\footnote{Corresponding author}
\\
{\small Software Engineering, LUT School of Engineering Science}\\
{\small Mukkulankatu 19, 15210 Lahti, Finland}\\
\small{\tt jouni.jarvinen@lut.fi}
\and
Michiro Kondo$^{\ast}$\\
\small{Department of Mathematics, School of System Design and Technology}\\
\small{Tokyo Denki University, Senju Asahi-cho 5, Adachi, Tokyo, Japan}\\
\small{\tt mkondo@mail.dendai.ac.jp}
}
\date{}

\begin{document}
\maketitle

\begin{abstract}
We consider fuzzy rough sets defined on De Morgan Heyting algebras.
We present a theorem that can be used to obtain several correspondence
results between fuzzy rough sets and fuzzy relations defining them.
We characterize fuzzy rough approximation operators corresponding to 
compositions of reflexive, transitive, mediate, Euclidean and adjoint fuzzy relations defined on De~Morgan Heyting algebras by using only a 
single axiom.

\smallskip\noindent%
\small{\textbf{Keywords:} rough set, $L$-fuzzy set, $L$-fuzzy relation, relational correspondence, De~Morgan Heyting algebra}
\end{abstract}

\section{Introduction}

Rough sets were introduced by Z.~Pawlak in \cite{Pawlak82} to deal with concepts that cannot be defined precisely in terms of our knowledge. In rough set theory, knowledge about objects $U$ is given in terms of an \emph{indistinguishability relation} $E$, which is an equivalence $E$ on $U$ interpreted so that two elements
are $E$-related if we cannot distinguish them in terms of the knowledge $E$.
For each subset, $X$ the \emph{lower approximation} $X_E$ is the set of 
elements whose $E$-equivalence class is included in $X$. 
The set $X_E$ is interpreted as the set of elements that certainly belong to $X$
in view of the knowledge $E$. The \emph{upper approximation}
$X^E$ consists of elements whose $E$-class intersects with $X$.  
The set $X^E$ can be seen as the set of elements which possibly belong to 
$X$. This means that every vague concept can be approximated from below and 
above  by two sets that are definable by the knowledge $E$.

In the literature, numerous studies can be found in which equivalences are replaced by an arbitrary binary relation reflecting, for instance, similarity or preference between the elements; see \cite{Orlowska1998, Yao96}, for example. 
Also in such a generalized setting, the lower approximation $X_R$ consists of 
elements which necessarily are in $X$ and $X^R$ is the set of elements which possibly are in $X$, in the view on knowledge $R$.

Fuzzy sets introduced by L.~Zadeh  \cite{Zadeh65} are generalizations of traditional
sets such that the set-memberships are expressed by a real number of the 
interval $[0,1]$. Similarly, Zadeh defined fuzzy relations as generalizations 
of traditional binary relations. Soon after Zadeh's paper, J.~A.~Goguen
\cite{Goguen1967} defined so-called $L$-fuzzy sets and $L$-fuzzy relations, 
in which the $[0,1]$-interval is replaced by a complete lattice $L$ having 
$0$ as the smallest and $1$ as the greatest elements. Since then, all kinds
of structures, such as Heyting algebras or residuated lattices, have been
presented as a basis for fuzzy sets \cite{Pavelka1979}.

The first approach to integrate rough set theory and fuzzy set theory is
the paper by  D.~Dubois and H.~Prade \cite{DuboisPrade1990}, where
they introduced fuzzy rough sets. A comprehensive evaluation of the most 
relevant fuzzy  rough set models proposed in the literature can be found in
\cite{Deer2015}.

Correspondences are understood as conditions that connect the properties of
relations to the properties of approximation operations. For instance,
if $R$ is an arbitrary binary relation on $U$, then $R$ is symmetric if
and only if $(X_R)^R \subseteq X \subseteq (X^R)_R$ holds for all subsets
$X$ of $U$. Correspondences for binary relations are well-studied in the literature. For instance,
reflexive, symmetric, and transitive relations are characterized in
\cite{Jarvinen2005}. In \cite{Yao96}, the authors considered serial,
reflexive, symmetric, transitive and Euclidean relations. 
Correspondences for serial, reflexive, mediate, transitive and alliance relations are give in \cite{Zhu2007}. However, there seems to be a problem in
the characterization of alliance relations, and we consider this issue
in Example~\ref{exa:alliance}.

As noted above, there are several ways to generalize rough sets to
fuzzy rough sets. This means that there are various types of correspondence
results in the literature depending on the setting. 
For instance, in \cite{radzikowska2004fuzzy} the authors
considered $L$-fuzzy rough approximations in the case $L$ is a
complete residuated lattice. They characterized serial, reflexive,
symmetric, Euclidean, and transitive $L$-fuzzy relations.
In \cite{WuZang2004}, the authors presented correspondences 
for reflexive, symmetric, transitive, and Euclidean $[0,1]$-fuzzy relations. 
Similar study of $[0,1]$-fuzzy rough sets can be found in \cite{Liu08}.
Some other recent types of fuzzy rough approximations and their correspondences
are considered in \cite{Jin23, Sun20, Sun23, Wei21, Xu23, Zhao21}, for example.

In \cite{Pang2019}, the authors consider such $L$-fuzzy relations that $L$ is a Heyting algebra provided with an antitone involution $'$. This 
means that $L$ is a so-called De~Morgan Heyting algebra.
Three new types (mediate, Euclidean, adjoint) of $L$-fuzzy rough approximation operators were characterized in that paper.
In this work, we adopt their approach. We will present a new uniform method in Theorem~\ref{thm:corr_gene} which
covers the results in \cite{Pang2019}, and several others.

L.~Liu \cite{Liu13} initiated the study of characterizing fuzzy rough set approximations by only one axiom. In \cite{Pang2019}, the authors provided an axiomatic approach to $L$-fuzzy rough approximation operators, where $L$
is a De~Morgan Heyting algebra. They also presented
single axioms to characterize $L$-fuzzy rough approximation operators corresponding to mediate, Euclidean, adjoint $L$-fuzzy relations as well as their compositions. In that work, they gave the following open problem: 
``Using single axioms to characterize $L$-fuzzy rough approximation operators corresponding to compositions
of serial, reflexive, symmetric, transitive, mediate, Euclidean and adjoint $L$-fuzzy relations.'' In this work, we present a general solution to this
problem, which covers also other types of relations and their compositions.

The paper is structured as follows. In Section~\ref{Sec:Preliminaries}
we recall from the literature notions and basic results used in this 
work. More precisely, in Section~\ref{subsec:crisp_relations}, we consider
the essential correspondence results related to (crisp)
binary relations and rough approximation operations. In this
work, we consider fuzzy rough sets on De~Morgan Heyting algebras, and Section~\ref{subsec:demorgan_heyting} covers the
basic facts of these types of algebras.

Section~\ref{sec:correspondences} is devoted to the correspondence results characterizing properties of $L$-fuzzy relations in terms of fuzzy rough approximation operators.
In Section~\ref{subsec:general} we present a result which is frequently
used in other subsections to prove correspondences. We present
correspondence for reflexive and symmetric (Section~\ref{subsec:refl_symm}),
transitive and mediate (Section~\ref{subsec:trans_mediate}), 
Euclidean and adjoint (Section~\ref{subsec:EucAdj}), and functional and positive alliance relations 
(Section~\ref{subsec:func_alliance}). There are
several ways to generalize serial relations to their fuzzy counterparts, and
in Section~\ref{subsec:serial} we consider three such generalizations.

Finally, in Section~\ref{Sec:Axiomatization} we solve the problem
given in \cite{Pang2019} and some concluding remarks end the work.

\section{Preliminaries} \label{Sec:Preliminaries}

In this section, we first recall correspondences between (crisp) binary 
relations and  rough approximation operations. 
The second part of this section considers
De~Morgan Heyting algebras and fuzzy rough approximation operations defined
on them.

\subsection{Binary relations and rough approximation operators}
\label{subsec:crisp_relations}

A \emph{binary relation} on $U$ is a set of ordered pairs $(x, y)$, 
where $x$ and $y$ are elements of $U$. If $x$ and $y$ are $R$-related,
this is denoted by $x \, R \, y$ or $(x,y) \in R$.

Let us introduce some general properties of binary relations. A binary relation $R$ on $U$ is
\begin{enumerate}[label = {(\roman*)}]
\item \emph{serial}, if for all $x \in U$, there is $y \in U$ such that $x \, R \, y$, 
\item \emph{reflexive}, if $x \, R \, x$ for all $x \in U$,
\item \emph{symmetric}, if $x \,R \, y$, then $y \, R \, x$ for all $x,y \in U$,
\item \emph{transitive}, if $x \, R \, y$ and $y \, R \, z$, then $x \, R \, z$ for all $x,y,z \in U$,
\item \emph{mediate}, if $x \, R \, y$ for some $x,y \in U$, there is $z \in U$ such that $x \, R \, z$ and $z \, R \, y$,
\item \emph{Euclidean}, if $x \, R \, y$ and $x \, R \, z$, then $y \, R \, z$ for all $x,y,z \in U$.
\end{enumerate}

Next, we recall (see e.g.~\cite{Jarvinen2005, Yao96}) rough approximation operations
defined by arbitrary relations. For any subset $X \subseteq U$, the 
\emph{lower approximation} of $X$ is defined as
\[ X_R = \{ x \in U \mid x \, R \, y \mbox{ implies } y \in X \}. \]
The \emph{upper approximation} of $X$ is 
\[ X^R = \{ x \in U \mid \emph{there exists $y \in X$ such that $x \, R \, y$ }\}. \]
The following correspondence results can be found in the literature; see \cite{Jarvinen2005,Yao96,Zhu2007}, for instance:
\begin{enumerate}[label = {(\roman*)}]
\item $R$ is serial $\iff U^R = U \iff \emptyset_R = \emptyset \iff (\forall X \subseteq U) \, X_R \subseteq X^R$,
\item $R$ is reflexive $\iff  (\forall X \subseteq U) \, X_R \subseteq X \iff (\forall X \subseteq U) \, X \subseteq X^R$,
\item $R$ is symmetric $\iff  (\forall X \subseteq U) \, (X_R)^R \subseteq X \iff (\forall X \subseteq U) \, X \subseteq (X^R)_R$,
\item $R$ is transitive $\iff (\forall X \subseteq U) \,
(X^R)^R \subseteq X^R \iff (\forall X \subseteq U) \,
X_R \subseteq (X_R)_R$.
\item $R$ is mediate $\iff (\forall X \subseteq U) \,
X^R \subseteq (X^R)^R \iff  (\forall X \subseteq U) \,
(X_R)_R \subseteq X_R$.
\item $R$ is Euclidean $\iff (\forall X \subseteq U) \,
X^R \subseteq (X^R)_R \iff  (\forall X \subseteq U) \,
(X_R)^R \subseteq X_R$.
\end{enumerate}
In this work, we consider similar correspondences in the setting of 
$L$-fuzzy rough sets, when a complete De~Morgan Heyting algebra
is defined on $L$.

\subsection{Approximation operators in De~Morgan Heyting algebras} 
\label{subsec:demorgan_heyting}

We begin by recalling some definitions from \cite{Castano11}.
A \emph{Heyting algebra} is an algebra 
\[ (H, \vee, \wedge, \Rightarrow, 0, 1)\] 
of type $(2,2,2,0,0)$ for which
$(H, \vee, \wedge, 0,1)$ is a bounded distributive lattice and $\Rightarrow$ is the operation of \emph{relative pseudocomplementation},
that is, for $a,b, c \in H$, $a \wedge c \leq b$ iff $c \leq a \Rightarrow b$.

A \emph{De~Morgan Heyting algebra} is an algebra 
\[ \mathbf{L} = (L, \vee, \wedge, \Rightarrow , ', 0, 1)
\] 
of type $(2, 2, 2, 1, 0, 0)$ such that 
$(L, \vee, \wedge, \Rightarrow, 0, 1)$ is a Heyting algebra and $'$ is an antitone involution, that is, 
$(x \wedge y)' = x' \vee y'$ and $x'' = x$ for all $x \in L$.

\begin{example} Let us consider the 4-element lattice 
$L := 0 < a,b < 1$ in which the elements $a$ and $b$ are incomparable. Then
$(L, \vee, \wedge, \Rightarrow, 0, 1)$ is a well-known Heyting algebra.

There are two ways to define the operation $'$ in $L$.
We can define 
\begin{center}
$0' = 1, \ a' = b, \ b' = a, \ 1' = 0$.
\end{center}
It is also possible to define
\begin{center}
$0' = 1, \ a' = a, \ b' = b, \ 1' = 0$.
\end{center}
In both cases, we have a De~Morgan Heyting algebra defined on $L$.
\end{example}

A complete lattice $L$ satisfies the \emph{join-infinite distributive law} if for any $S \subseteq L$ and $x \in L$,
\begin{equation*}\label{Eq:JID} \tag{JID}
x \wedge \big ( \bigvee S \big ) = \bigvee \{ x \wedge y \mid y \in S \}.
\end{equation*}
The dual condition is the \emph{meet-infinite distributive law}, (MID).  It is well known that a complete lattice defines a Heyting algebra
if and only if it satisfies (JID). Each De~Morgan Heyting algebra $\mathbf{L}$ defined on a complete lattice $L$ satisfies both (JID) and (MID),
because $'$ is an order-isomorphism between $(L,\leq)$ and its dual $(L,\geq)$. In this work, De~Morgan Heyting algebras defined on a complete lattice
are called \emph{complete De~Morgan Heyting algebras}.

Let $\mathbf{L}$ be a complete De~Morgan Heyting algebra and $U$ a universe. An \emph{$L$-fuzzy set} on $U$ is a mapping $A \colon U \to L$.
We often drop the word `fuzzy' and speak about $L$-sets. The family of all $L$-sets on $U$ is denoted by $\mathcal{F}_L(U)$.

The set $\mathcal{F}_L(U)$ may be ordered \textit{pointwise} by setting for $A,B \in \mathcal{F}_L(U)$, $A \leq B$ if and only if  $A(x) \leq B(x)$ for all $x \in U$.
If $\mathbf{L}$ is a complete De~Morgan Heyting algebra, then $\mathcal{F}_L(U)$ forms a complete De~Morgan Heyting algebra such that for all
$\{A_i\}_{i \in I} \subseteq \mathcal{F}_L(U)$ and $x \in U$,
\[
\big ({\bigvee}_i A_i \big ) (x)  = {\bigvee}_i  A_i(x) \quad \text{ and } \quad
\big ({\bigwedge}_i A_i \big ) (x)  = {\bigwedge}_i  A_i(x)
\]
The operations $\Rightarrow$ and $'$ are defined for $A,B \in \mathcal{F}_L(U)$ and $x \in U$ by
\[ 
(A \Rightarrow B)(x) = A(x) \Rightarrow B(x) \qquad  \text{ and } \qquad A'(x) = A(x)' .
\] 
The map $\mathbf{0} \colon x \mapsto 0$ is the smallest and  $\mathbf{1} \colon x \mapsto 1$ is the greatest element of $\mathcal{F}_L(U)$, respectively.

An \emph{$L$-fuzzy relation $R$ on $U$} is a mapping $U \times U \to L$. We often use the term ``$L$-relation'' instead of ``$L$-fuzzy relation''.
The following definition of $L$-approximations can be found in \cite{Pang2019}.

\begin{definition}
Let $\mathbf{L}$ be a complete De~Morgan Heyting algebra, $R$ an $L$-relation on $U$ and $A \in \mathcal{F}_L(U)$. The
\emph{upper $L$-fuzzy approximation} and \emph{lower $L$-fuzzy approximation of $A$} are defined by  
\begin{align*}
\U(A)(x) = {\bigvee}_y \big ( R(x,y) \wedge A(y) \big ) \quad \text{ and } \quad \L(A)(x) = {\bigwedge}_y \big ( R(x,y)' \vee A(y) \big ),
\end{align*}  
respectively.
\end{definition}
\noindent%
If there is no danger of confusion, we may denote $\U(A)$ and $\L(A)$ simply by $\U A$ and $\L A$. In addition, 
$L$-fuzzy approximations are called simply $L$-approximations.

In \cite{Pang2019}, the following properties of $L$-approximations generalizing well-known properties of crisp rough approximations are proved.

\begin{proposition} \label{Prop:Basic} Let $U$ be a set, $L$ a complete De~Morgan Heyting algebra and $R$ an $L$-relation on $U$.
For $\{A_i\}_{i \in I} \subseteq \mathcal{F}_L(U)$, $A \in \mathcal{F}_L(U)$ and $x \in U$, the following assertions hold:
\begin{enumerate}[label = {\rm (\arabic*)}]
\item $\L\mathbf{0} = \mathbf{0}$ \  and  \ $\U \mathbf{1} = \mathbf{1}$;
\item $(\U A)' = \L(A')$  \ and  \ $(\L A)' = \U(A')$;
\item $\U({\bigvee}_i A_i) = {\bigvee}_i   \U(A_i)$ \ and  \quad $\L ({\bigwedge}_i A_i ) = {\bigwedge}_i  \L(A_i)$;
\item $\U( {\bigwedge}_i A_i) \leq {\bigwedge}_i  \U(A_i)$ \ and  \ $\L( {\bigvee}_i A_i ) \geq {\bigvee}_i   \L(A_i)$.
\end{enumerate}
\end{proposition}

For $a\in L$, we define a `constant' $L$-set $\overline{a}$ by 
\[ \overline{a}(x) = a \text{ for all $x \in U$}.\]
This means that  $\overline{0} = \mathbf{0}$ and  $\overline{1} = \mathbf{1}$. For any $x \in U$, we define a map $I_x$ by
\[
I_x(y) = 
\begin{cases}
1 & \text{if $x =y$},\\
0 & \text{otherwise}.
\end{cases} 
\]
The idea is that the map $I_x$ corresponds to the singleton $\{x\}$. 
For any $x,y \in U$, $\U(I_y)(x) = \bigvee_z (R(x,z) \wedge I_y(z))$. Because $I_y(z) = 1$ iff $z = y$, we obtain the following
equality, which will be used frequently in this work:
\begin{equation} \label{eq:Upper_I_y}
  (\forall x,y \in  U) \, \U(I_y)(x) = R(x,y) .
\end{equation}

It is noted in \cite{Pang2019} that each $A \in \mathcal{F}_L(U)$ can be written in
two ways:
\begin{equation}\label{eq:widehat}
A = {\bigvee}_x \big ( \overline{A(x)} \wedge I_x \big )  = {\bigwedge}_x \big ( \overline{A(x)} \vee (I_x)' \big ). 
\end{equation}
Note that $(I_x)'$ corresponds the set-theoretical complement $U \setminus \{x\}$ of the singleton $\{x\}$, and
the latter equality is clear by Lemma~{3.17} of \cite{Pang2019}.
The following facts were also proved in \cite{Pang2019}.

\begin{lemma} \label{lem:bar_prop}
Let $L$ be a complete De~Morgan Heyting algebra and $R$ an $L$-relation on $U$. For all $a \in L$ and $A \in \mathcal{F}_L(U)$,
\begin{enumerate}[label = {\rm (\arabic*)}]
\item $\U(\overline{a}) \leq \overline{a} \leq \L(\overline{a})$;
\item $\U(\overline{a} \wedge A) = \overline{a} \wedge \U(A)$ \ and \ $\L(\overline{a} \vee A) = \overline{a} \vee \L(A)$;
\item $\U(\overline{a} \vee A) \leq \overline{a} \vee \U(A)$ \ and \ $\L(\overline{a} \wedge A) \geq \overline{a} \wedge \L(A)$.
\end{enumerate}
\end{lemma}

\section{Correspondence results}
\label{sec:correspondences}

In this section, we assume that $\mathbf{L}$ is a complete De~Morgan Heyting algebra, $U$ is a universe, and $R$ is an $L$-relation on $U$. 

\subsection{A general result}
\label{subsec:general}

In this subsection, we present a general result which can be used to
obtain several correspondence results. Let $\mathcal{S}$ be a finite 
combination of rough approximation operators $\L$ and $\U$.
Because the operators $\L$ and $\U$ are order-preserving, the operator $\mathcal{S}$ is order-preserving.
From this it follows that
\[ {\bigvee}_i \mathcal{S}(A_i) \leq \mathcal{S} \big ( {\bigvee}_i A_i \big ) \ \text{ and } \
   {\bigwedge}_i \mathcal{S}(A_i) \geq \mathcal{S} \big ( {\bigwedge}_i A_i \big ) . \]

\begin{lemma} \label{lem:step}
For all $a \in L$,
\[
\bar{a} \wedge \mathcal{S}(I_x) \leq \mathcal{S}(\bar{a} \wedge I_x).
\]
\end{lemma}

\begin{proof} We prove the claim by induction.
If $n = 1$, then the two cases $\bar{a} \wedge \U(I_x) \leq \U(\bar{a} \wedge I_x)$ and $\bar{a} \wedge \L(I_x) \leq \L(\bar{a} \wedge I_x)$
are clear by Lemma~\ref{lem:bar_prop}(2) by setting $A = I_x$.

Suppose that the claim holds for all combinations consisting of $n$ $\L$ and $\U$ operators. Let $\mathcal{S}$ be a combination $n + 1$
operators. Then $\mathcal{S} = \U \circ \mathcal{S}_1$ or $\mathcal{S} = \L \circ \mathcal{S}_2$, where $\mathcal{S}_1$ and $\mathcal{S}_2$ are combinations
of $\U$ and $\L$ of length $n$.

If $\mathcal{S} = \U \circ \mathcal{S}_1$, then
\[ \overline{a} \wedge \mathcal{S}(I_x) = \overline{a} \wedge \U(\mathcal{S}_1(I_x)) = \U(\overline{a} \wedge \mathcal{S}_1(I_x)) \leq
\U(\mathcal{S}_1(\bar{a} \wedge I_x)) = \mathcal{S}(\bar{a} \wedge I_x).
\]

If $\mathcal{S} = \L \circ \mathcal{S}_2$, then
\[ \overline{a} \wedge \mathcal{S}(I_x) = \overline{a} \wedge \L(\mathcal{S}_2(I_x)) \leq \L(\overline{a} \wedge \mathcal{S}_2(I_x)) \leq
\L(\mathcal{S}_2(\bar{a} \wedge I_x)) = \mathcal{S}(\bar{a} \wedge I_x). 
\qedhere
\]
\end{proof}

\begin{theorem} \label{thm:corr_gene}
Let $\mathcal{S}$ be a finite combination of rough approximation operators $\L$ and $\U$.
If $\U(I_x) \leq \mathcal{S}(I_x)$ for all $x \in U$, then $\U(A) \leq \mathcal{S}(A)$ for all $A \in \mathcal{F}_L(A)$.
\end{theorem}

\begin{proof} Assume that $\U(I_x) \leq \mathcal{S}(I_x)$ for all $x \in U$. Then,
\begin{align*}
\U(A) & = \U \big ({\bigvee}_x (\overline{A(x)} \wedge I_x) \big)            & \text{(by \eqref{eq:widehat})}  \\
  & = {\bigvee}_x \U(\overline{A(x)} \wedge I_x)                             & \text{(by Prop.~\ref{Prop:Basic}(3))}  \\
  & = {\bigvee}_x \big (\overline{A(x)} \wedge \U(I_x) \big)              & \text{(by Lemma \ref{lem:bar_prop}(2))} \\
  & \leq {\bigvee}_x \big (\overline{A(x)} \wedge \mathcal{S}(I_x) \big)  & \text{(by assumption)} \\
  & \leq {\bigvee}_x \mathcal{S}(\overline{A(x)} \wedge I_x )                & \text{(by Lemma~\ref{lem:step})} \\
  & \leq \mathcal{S} \big ({\bigvee}_x \overline{A(x)} \wedge I_x \big )     & \text{($\mathcal{S}$ preserves $\leq$)} \\
  & \leq \mathcal{S}(A)                                                      &  \text{(by \eqref{eq:widehat})} 
\end{align*} 
\end{proof}

\subsection{Reflexive and symmetric relations}
\label{subsec:refl_symm}

Let us start with the following definition.

\begin{definition}
An $L$-relation $R$ is \emph{reflexive}, if 
$(\forall x \in U) \, R(x,x)=1$.
\end{definition}

The following lemma gives a characterization of reflexivity.

\begin{lemma} \label{lem:refl}
An $L$-relation $R$ is reflexive if and only if for all $x \in U$, 
$I_x \leq \U({I_x})$.
\end{lemma}

\begin{proof} 
If $R$ is reflexive, $\U({I_x})(x) = R(x,x) = 1 = I_x(x)$ for all $x \in U$. If $R$ is not reflexive, then there is $x \in U$ such that $R(x,x) \neq 1$.
Now, $I_x(x) = 1$ and $\U(I_x)(x) = R(x,x) \neq 1$, giving $I_x \nleq \U({I_x})$.
\end{proof}

\begin{proposition} \label{prop:reflexive}
Let $\mathbf{L}$ complete De~Morgan Heyting algebra and $R$ an $L$-relation on $U$.
Then, the following are equivalent:
\begin{enumerate}[label = {\rm (\arabic*)}]
\item $R$ is reflexive;
\item $(\forall A \in \mathcal{F}_L(U)) \, A \leq \U A$;
\item $(\forall A \in \mathcal{F}_L(U)) \, \L A \leq A$.
\end{enumerate}
\end{proposition}

\begin{proof}
Suppose that $R$ is reflexive. Then, by Lemma~\ref{lem:refl}, $I_x \leq \U({I_x})$ for all $x \in U$.
We have that for $A \in \mathcal{F}_L(U)$,
\begin{align*} \textstyle
  A &= {\bigvee}_x \Big ( \overline{A(x)} \wedge I_x \Big ) \leq {\bigvee}_x \Big ( \overline{A(x)} \wedge \U({I_x}) \Big )
  = {\bigvee}_x \U \Big ( \overline{A(x)} \wedge I_x \Big ) \\
  & =  \U \Big ( {\bigvee}_x ( \overline{A(x)} \wedge {I_x}) \Big ) = 
  \U A.
\end{align*}
Thus, (1) implies (2). Suppose that (2) holds. Then for all $x \in U$,
$I_x \leq \U({I_x})$, which by Lemma~\ref{lem:refl} yields that $R$ is reflexive. Thus, also (2) implies (1).

We prove that (2) and (3) are equivalent. Assume (2) holds. Then for any $A \in \mathcal{F}_L(U)$,  $A' \leq \U(A') = (\L A)'$, which implies $\L A \leq A$.
Similarly, if (3) holds, then $\L(A') \leq A'$ implies $A \leq (\L(A'))' = \U(A)$.
\end{proof}

A binary relation $\rho$ is symmetric whenever, for all $x,y \in U$, 
$x \, \rho \, y$ implies $y \, \rho \, x$. The symmetry condition can be expressed in the form
$(x \, \rho^c \, y) \vee (y \rho \, x)$, where $\rho^c$ denotes the complement of the relation $\rho$, that is,
$\rho^c = (U \times U) \setminus \rho$. Based on this fact, we
present the following definition.

\begin{definition} \label{def:symmetry}
An $L$-relation $R$ \emph{symmetric} if
$(\forall x,y \in U) \, R(x,y)' \vee R(y,x) = 1$.
\end{definition}
  
We can now express symmetry in terms of $I_x$ and $L$-approximations.

\begin{lemma} \label{lem:symm}
An $L$-relation $R$ is symmetric if and only if for all $x \in U$, 
$I_x \leq \L \U (I_x)$.
\end{lemma}

\begin{proof} For all $x \in U$,
  \[ \L \U (I_x)(x) = {\bigwedge}_y ( R(x,y)' \vee \U (I_x)(y) ) = {\bigwedge}_y (R(x,y)' \vee R(y,x)).\]
If $R$ is symmetric, then $\L \U (I_x)(x) = 1$, which means that
$\L \U (I_x)(x) \geq I_x (x)$. On the other hand,
if  $ \L \U (I_x) \geq I_x$, then $ \L \U (I_x)(x) = 1$ and 
$\bigwedge_y (R(x,y)' \vee R(y,x)) = 1$.
This gives that for each $y$, $R(x,y)' \vee R(y,x) = 1$.
\end{proof}

We can now write the following characterization of symmetric $L$-relations.

\begin{proposition} Let $\mathbf{L}$ complete De~Morgan Heyting algebra and $R$ an $L$-relation on $U$. Then, the following are equivalent:
\begin{enumerate}[label = {\rm (\arabic*)}]
\item $R$ is symmetric;
\item $(\forall A \in \mathcal{F}_L(U)) \, A \leq \L \U A$;
\item $(\forall A \in \mathcal{F}_L(U)) \, \U \L A \leq A$.
\end{enumerate}
\end{proposition}

\begin{proof}
Suppose that $R$ is symmetric. Then, by Lemma~\ref{lem:symm}, 
$I_x \leq \L \U (I_x)(x)$ for each $x \in U$. We have that
\begin{align*} 
A & = {\bigvee}_x \big (\overline{A(x)} \wedge I_x \big) \leq {\bigvee}_x \big (\overline{A(x)} \wedge  \L \U (I_x) \big )
\leq {\bigvee}_x \L \big (\overline{A(x)} \wedge \U ({I_x}) \big ) \\
& = {\bigvee}_x \L \U \big ( \overline{A(x)} \wedge {I_x} \big )
\leq \L \U \Big ({\bigvee}_x (\overline{A(x)} \wedge {I_x}) \Big ) = \L \U A.
\end{align*}
This means that (1) implies (2). If we set $A = I_x$ in (2), we get that $I_x \leq \L \U (I_x)$ for $x \in U$.
Hence, $R$ is symmetric. The equivalence of (2) and (3) follows easily by the duality of approximation operations.
\end{proof}

We end this subsection by the following remark.

\begin{remark}
For a fuzzy relation $R$, symmetry is typically expressed by a condition 
\begin{equation} \tag{S*} \label{eq:symm1}
 (\forall x,y \in U) \, R(x,y) = R(y,x).
 \end{equation}
Let us briefly consider how this relates to the symmetry of 
Definition~\ref{def:symmetry}. 

The three-element chain $\mathbf{3} := 0 < u < 1$ is a well-known
Heyting algebra. The operation
$'$ is defined by $0' = 1$, $u' = u$, and $1' = 0$. Assume that
$U = \{x,y\}$. Let us define a $\mathbf{3}$-fuzzy relation $R$ on $U$ in such a way that $R(x,y) = R(y,x) = u$, $R(x,x) = 0$ and $R(y,y) = 1$. 
Now the relation satisfies \eqref{eq:symm1}. 
But $R(x,y)' \vee R(y,x) = u \vee u = u$, so \eqref{eq:symm1} does 
not imply symmetry in the sense of Definition~\ref{def:symmetry}.

On the other hand, consider the 4-element Heyting algebra
$L := 0 < a,b < 1$, where $a$ and $b$ are incomparable. 
Let the operation $'$ be defined by $0' = 1$, $a' = a$,
$b' = b$, $1' = 0$. Assume $U = \{x,y\}$ and define the $L$-relation 
$R$ by $R(x,x) = 1$, $R(y,y) = 0$, $R(x,y) = a$, $R(y,x) = b$.
Now $R$ is symmetric in the sense of Definition~\ref{def:symmetry}, 
but  $R(x,y) \neq R(y,x)$, that is,
\eqref{eq:symm1} does not hold.

It remains an \emph{open problem} whether \eqref{eq:symm1} be characterized in terms of fuzzy rough approximation operations.
\end{remark}

\subsection{Transitive and mediate relations}
\label{subsec:trans_mediate}

\begin{definition} \label{def:transitive}
An $L$-relation is \emph{transitive}, if 
$(\forall x,y,z \in U) \, R(x, z) \wedge R(z, y) \leq R(x, y)$. 
\end{definition}

Obviously, transitivity is equivalent to that $(\forall x,y \in U) \,\bigvee_z (R(x, z) \wedge R(z, y)) \leq R(x, y)$.

\begin{lemma} \label{lem:trans}
An $L$-relation $R$ is transitive if and only if for all $x \in U$, $\U \U (I_x) \leq \U (I_x)$.
\end{lemma}

\begin{proof} For all $x \in U$,
\[ \U \U (I_x)(y) = {\bigvee}_z \big (R(y,z) \wedge \U (I_x)(z) \big ) = {\bigvee}_z (R(y,z) \wedge R(z,x)) ,\] 
which is below $R(y,x) = \U(I_x)(y)$ by the transitivity of $R$. 

On the other hand, for all $x,y \in U$,
\[ \U \U (I_y)(x) = {\bigvee}_z ( R(x,z) \wedge R(z,y) ) \quad \text{and} \quad \U (I_y)(x) = R(x,y).\]
If $\U \U (I_y) \leq \U (I_y)$, then the transitivity condition holds for $x$ and $y$, which completes the proof. 
\end{proof}

The next proposition characterizes transitive $L$-relations in terms of approximation operations.

\begin{proposition} \label{prop:transitive}
Let $\mathbf{L}$ be a complete De~Morgan Heyting algebra and $R$ an $L$-relation on $U$. Then, the following are equivalent:
\begin{enumerate}[label = {\rm (\arabic*)}]
\item $R$ is transitive;
\item $(\forall A \in \mathcal{F}_L(U)) \, \U \U A \leq \U A$;
\item $(\forall A \in \mathcal{F}_L(U)) \, \L A \leq \L \L A$.
\end{enumerate}
\end{proposition}

\begin{proof} Assume that $R$ is transitive. By Lemma~\ref{lem:trans}, $\U \U (I_x) \leq \U (I_x)$ for $x \in U$. We get that
\begin{align*}
\U \U A &=  \U \U \Big ( \Big ({\bigvee}_x (\overline{A(x)} \wedge I_x) \Big) \Big ) = \U \Big ( {\bigvee}_x \U ( \overline{A(x)} \wedge I_x) \Big )
= \U \Big ( {\bigvee}_x ( \overline{A(x)} \wedge \U (I_x) \Big ) \\
& =  \Big ( {\bigvee}_x ( \overline{A(x)} \wedge \U \U (I_x) \Big ) 
\leq {\bigvee}_x (\overline{A(x)} \wedge \U (I_x)) = 
\U \Big ({\bigvee}_x (\overline{A(x)} \wedge I_x) \Big ) \\
& = \U A.
\end{align*}
Thus, (2) holds. Conversely, if (2) holds, then $\U \U (I_x) \leq \U (I_x)$ for all $x \in U$ and $R$ is transitive.
The equivalence of (2) and (3) is clear.
\end{proof}

We recall the following definition from \cite{Pang2019}.
\begin{definition}
An $L$-relation $R$ is called \emph{mediate} if 
$(\forall x,y \in U) \, R(x,y) \le \bigvee_z ( R(x,z) \wedge R(z,y) )$.
\end{definition}

\begin{lemma} \label{lem:mediate}
An $L$-relation $R$ is mediate if and only if for all $x \in U$, 
$\U \U (I_x) \geq \U (I_x)$.
\end{lemma}

\begin{proof} For all $x,y \in U$,
\[
\U (I_y)(x) \le \U \U (I_y)(x) \iff R(x,y) \le {\bigvee}_z (R(x,z) \wedge R(z,y)), 
\]
which completes the proof.
\end{proof}

We may now write the following proposition.

\begin{proposition} \label{prop:mediate}
Let $\mathbf{L}$ be a complete De~Morgan Heyting algebra and $R$ an $L$-relation on $U$. Then, the following are equivalent:
\begin{enumerate}[label = {\rm (\arabic*)}]
\item $R$ is mediate;
\item $(\forall A \in \mathcal{F}_L(U)) \, \U A \leq \U \U A$;
\item $(\forall A \in \mathcal{F}_L(U)) \, \L \L A \leq \L A$.
\end{enumerate}
\end{proposition}

\begin{proof} Assume that $R$ is mediate. By Lemma~\ref{lem:mediate}, $\U (I_x) \leq \mathcal{S} (I_x) = \U \U (I_x)$ for all $x \in U$, 
where $\mathcal{S}$ denotes the  combination $\U \U$. Therefore,
$\U A \leq \mathcal{S}(A) = \U \U A$ for all $A \in \mathcal{F}_L(U)$
follows directly from Theorem~\ref{thm:corr_gene}. Therefore, (2) implies (1). By applying $A = I_x$, it is immediate that (2) implies (1). 
 
The equivalence of (2) and (3) can be proved as in our earlier proofs. 
\end{proof}

\subsection{Euclidean and adjoint relations}
\label{subsec:EucAdj}

A binary relation $\rho$ is Euclidean if $x \, \rho \, y$ and $x \, \rho \, z$ imply $y \, \rho \, z$. Obviously,
$x \, \rho \, y$ and $x \, \rho \, z$ imply also $z \, \rho \, y$. 
As noted in \cite{Pang2019}, this is equivalent to
that $x \, \rho \, z$ and $z \, \rho^c \, y$ imply $x \, \rho^c \, y$. 

\begin{definition}
An $L$-relation $R$ on $U$ is called \emph{Euclidean} if 
\[ (\forall x,y,z \in U) \, R(x,y)' \geq  R(x,z) \wedge R(z,y)'.\] 
\end{definition}
\noindent%
Obviously, being Euclidean is equivalent to the condition
\begin{equation} \label{eq:euc} \tag{euc}
(\forall x,y \in L) \, R(x,y)' \geq {\bigvee}_z (R(x,z) \wedge R(z,y)' ) .
\end{equation}

\begin{lemma} \label{lem:Euc}
An $L$-relation $R$ is Euclidean if and only if for all $x \in U$, $\U (I_x) \leq \L \U (I_x)$.
\end{lemma}

\begin{proof} Let $x,y \in U$. As we have noted, $\U(I_y)(x) = R(x,y)$ and
\[ \L\U(I_y)(x) = {\bigwedge}_z (R(x,z)' \vee \U(I_y)(z)) = 
{\bigwedge}_z (R(x,z)' \vee R(z,y)).\] 
If $R$ is Euclidean, then $R(x,y)' \geq \bigvee_z (R(x,z) \wedge R(z,y)' )$, which implies
\begin{align*}
R(x,y) & = R(x,y)'' \leq \big ( {\bigvee}_z (R(x,z) \wedge R(z,y)' ) \big)' =
{\bigwedge}_z (R(x,z) \wedge R(z,y))' \\
& = {\bigwedge}_z (R(x,z)' \vee R(z,y)'') = {\bigwedge}_z (R(x,z)' \vee R(z,y)),
\end{align*}
This means $\U(I_y) \leq \L\U(I_y)$.

On the other hand, if $\U(I_y) \leq \L\U(I_y)$, then 
for all $x \in U$, $\U(I_y)(x) \leq \L \U (I_y)(x)$ and thus,
\[ R(x,y) \leq {\bigwedge}_z (R(x,z)' \vee R(z,y)).\]
Applying De Morgan operation $'$ for the both sides of the relation, we
obtain that \eqref{eq:euc} holds. 
\end{proof}

The next proposition characterizes Euclidean $L$-relations in terms of approximations.

\begin{proposition} \label{prop:Eucl}
Let $\mathbf{L}$ be a complete De~Morgan Heyting algebra and $R$ an $L$-relation on $U$. Then, the following are equivalent:
\begin{enumerate}[label = {\rm (\arabic*)}]
\item $R$ is Euclidean;
\item $(\forall A \in \mathcal{F}_L(U)) \, \U A \leq \L\U A$;
\item $(\forall A \in \mathcal{F}_L(U)) \, \U \L A \leq \L A$.
\end{enumerate}
\end{proposition}

\begin{proof} If $R$ is Euclidean, then $\U (I_x) \leq \L \U (I_x)$ for all $x \in U$. By Theorem~\ref{thm:corr_gene},
$\U A \leq \L \U A$ for all $A \in \mathcal{F}_L(U)$. Thus, (2) implies (1). By applying $A = I_x$, we see that (2) implies (1). 
The equivalence of (2) and (3) is obvious.
\end{proof}

We recall from \cite{Pang2019} the definition of adjoint relations.

\begin{definition}
An $L$-relation $R$ on $U$ is \emph{adjoint} if 
\[ (x,y \in U) \,
R(x,y)' \ge {\bigwedge}_z {\bigvee}_{w\neq y} (R(x,z)' \vee R(z,w)).
\]
\end{definition}

We can now characterize adjoint relations in terms of $I_x$ and
fuzzy approximations operations.

\begin{lemma} \label{lem:adjoint}
An $L$-relation $R$ is adjoint if and only if for all $x \in U$, $\U (I_x) \leq \U \L (I_x)$.
\end{lemma}

\begin{proof}
For all $x,y \in U$, we have the following equality:
\begin{align*}
  {\bigwedge}_{z} {\bigvee}_{w\neq y} (R(x,z)' \vee R(z,w)) 
  & = {\bigwedge}_{z} \big ( R(x,z)' \vee {\bigvee}_{w\neq y} R(z,w) \big ) \\
  & = {\bigwedge}_{z} \big ( R(x,z)' \vee {\bigvee}_{w} (R(z,w) \wedge {I_y}'(w)  \big) \\
  & = {\bigwedge}_{z} \big ( R(x,z)' \vee \U({I_y}')(z)  \big )\\
  & = \L \U({I_y}')(x) \\
  & = (\U \L ({I_y})(x))'.  
\end{align*}
If $R$ is adjoint, then
\[ (\U (I_y)(x))' = R(x,y)' \geq ( \U \L({I_y})(x)) ',\]
which is equivalent to $\U (I_y)(x) \leq \U \L ({I_y})(x)$. Hence,  $\U(I_y) \leq \U \L ({I_y})$ for all $y \in U$.
Conversely, if $\U (I_y) \leq  \U \L (({I_y})$ for all $y \in U$, 
then by the above equality, $R$ is adjoint.
\end{proof}

We can now extend Lemma~\ref{lem:adjoint} to the following correspondence
for adjoint relations using our Theorem~\ref{thm:corr_gene}.

\begin{proposition} \label{prop:adjoint}
Let $\mathbf{L}$ be a complete De~Morgan Heyting algebra and $R$ an $L$-relation on $U$. Then, the following are equivalent:
\begin{enumerate}[label = {\rm (\arabic*)}]
\item $R$ is adjoint;
\item $(\forall A \in \mathcal{F}_L(U)) \, \U A \leq \U \L A$;
\item $(\forall A \in \mathcal{F}_L(U)) \, \L \U A \leq \L A$.
\end{enumerate}
\end{proposition}

\begin{proof} If $R$ is adjoint, then $\U (I_x) \leq  \U \L (I_x)$ for all $x \in U$. Theorem~\ref{thm:corr_gene} gives that
$\U A \leq \U \L A$ for all $A \in \mathcal{F}_L(U)$. Thus, (1) implies (2). By applying $A = I_x$, we see that (2) implies (1). 
 
The equivalence of (2) and (3) is clear by the duality.
\end{proof}

\subsection{Functional and alliance relations}
\label{subsec:func_alliance}

A binary relation $\rho$ is functional when each $x \in U$ is $\rho$-related to at most one element. 
This means that $x \, \rho \, y$ implies that for all 
$z \in (U \setminus \{y\})$, $x \, \rho^c \, y$. 
Based on this, we present the following definition.

\begin{definition}\label{def:functional}
Thus, an $L$-relation $R$ on $U$ is \emph{functional} if 
\[ (\forall x,y) \, R(x,y) \leq  \bigwedge_{z \neq y} R(x,z)' .\]
\end{definition}
\noindent%
The condition means that $R(x,y) \leq  R(x,z)'$ holds for all
$x,y \in U$ and $z \neq y$.

\begin{lemma} 
An $L$-relation $R$ is functional if and only if for all $x \in U$, $\U (I_x) \leq \L (I_x)$.
\end{lemma}

\begin{proof} For all $x,y \in U$, $\U (I_y)(x) = R(x,y)$ and
\begin{align*}
  \L (I_y)(x) &= {\bigwedge}_z (R(x,z)' \vee (I_y)(z)) = 
  \Big ( {\bigwedge}_{z \neq y} (R(x,z)' \vee 0) \Big ) \wedge (R(x,y)' \vee 1) \\
  &= \Big ( {\bigwedge}_{z \neq y} R(x,z)' \Big ) \wedge 1 = 
  {\bigwedge}_{z \neq y} R(x,z)' .
\end{align*}
Because for all $y \in U$, $\U (I_y) \leq \L (I_y)$ is equivalent to that $\U (I_y)(x) \leq \L (I_y)(x)$ for all $x,y \in U$, the claim is proved. 
\end{proof}

\begin{proposition} \label{prop:func}
Let $\mathbf{L}$ be a complete De~Morgan Heyting algebra and $R$ an $L$-relation on $U$. Then, the following are equivalent:
\begin{enumerate}[label = {\rm (\arabic*)}]
\item $R$ is functional;
\item $(\forall A \in \mathcal{F}_L(U)) \, \U A \leq \L A$.
\end{enumerate}
\end{proposition}

\begin{proof}
If $R$ is functional, then $\U (I_x) \leq \L (I_x)$ for all $x \in U$. By Theorem~\ref{thm:corr_gene}, we obtain
$\U A \leq \L A$ for all $A \in \mathcal{F}_L(U)$.
Conversely, if $\U A \leq \L A$ for any $A \in \mathcal{F}_L(U)$, then $\U (I_x) \leq \L (I_x)$ for all $x \in U$. 
\end{proof}

`Positive alliance' relations were defined in \cite{Zhu2007} by stating that a binary relation $\rho$ is a
positive alliance if for any elements $x,y \in U$ such that $x \, \rho^c \, y$, there is $z \in U$ satisfying
$x \, \rho \, z$, but $z \, \rho^c \, y$.

It is clear that each reflexive relation is positive alliance, because if $a \, \rho^c \, b$, then
$a \, \rho \, a$ and $a \, \rho^c \, b$ hold trivially.

The following facts can be found in \cite{Ma2015}, but for the sake of completeness we give a proof.

\begin{lemma} Let $\rho$ be a serial and transitive binary relation on $U$. 
Then, the following conditions hold:
\begin{enumerate}[label = {\rm (\arabic*)}]
\item $\rho$ is a positive alliance;
\item for all $X \subseteq U$, $(X^\rho)_\rho \subseteq X^\rho$.
\end{enumerate}
\end{lemma}

\begin{proof} (1) Suppose $a \, \rho^c \, b$. Because $\rho$ is serial, there is $c$ such that $a \, \rho \, c$.
Now $c \, \rho \, b$ is not possible, because that would imply $a \, \rho \, b$, contradicting  $a \, \rho^c \, b$.
Thus, $c \, \rho^c \, b$.

(2) If $\rho$ is a serial binary relation on $U$, then $X_\rho \subseteq X^\rho$ for all $X \subseteq U$. In particular,
$(X^\rho)_\rho \subseteq (X^\rho)^\rho$. If $\rho$ is also transitive, then $(X^\rho)^\rho \subseteq X^\rho$.
Thus, if $\rho$ is serial and transitive,  $(X^\rho)_\rho \subseteq X^\rho$.
\end{proof}

On the other hand, it is clear that if $\rho$ is a positive alliance, then it is serial. However, there are
positive alliance relations that are not transitive.

\begin{example}
Let $U = \{a,b,c\}$ and let $\rho = \{ (a,b), (b,b), (c,a), (c,c) \}$. Then $\rho$ is serial, but not transitive
because $c \, \rho \, a$ and $a \, \rho \, b$, but $c \, \rho^c \, b$. Now we have that:
\begin{itemize}
\item $a \, \rho^c a$, and there is $b$ such $a \, \rho \, b$ and $b \, \rho^c \, a$;
\item $a \, \rho^c c$, and there is $b$ such $a \, \rho \, b$ and $b \, \rho^c \, c$;
\item $b \, \rho^c a$, and there is $b$ such $b \, \rho \, b$ and $b \, \rho^c \, a$;
\item $b \, \rho^c c$, and there is $b$ such $b \, \rho \, b$ and $b \, \rho^c \, c$;
\item $c \, \rho^c b$, and there is $c$ such $c \, \rho \, c$ and $c \, \rho^c \, b$.  
\end{itemize}
Hence, $\rho$ is a positive alliance.
\end{example}

It is clear that for all $x,y \in U$, $y \in \{x\}^\rho \iff y \, \rho \, x$. Therefore,
\[ y \in (\{x\}^\rho)^c \iff y \, \rho^c \, x \qquad \text{and} \qquad y \in  ( (\{x\}^\rho )^c )^\rho \iff
(\exists z) y \, \rho \, z  \ \& \ z \, \rho^c \, x .\]
As proved in \cite{Zhu2007}, $\rho$ is a positive alliance if and only if $(\{x\}^\rho)^c \subseteq ( (\{x\}^\rho  )^c  )^\rho$
for all $x \in U$. Note that the latter condition is equivalent to that $(\{x\}^\rho)_\rho \subseteq \{x\}^\rho$ for all $x \in U$.

\begin{example} \label{exa:alliance}
It is commonly accepted (see \cite{Ma2015,Zhu2007}, for instance) that for any binary relation $\rho$, the following are equivalent:
\begin{enumerate}[label = {\rm (\arabic*)}]
\item $\rho$ is a positive alliance;
\item for all $X \subseteq U$, $(X^\rho)_\rho \subseteq X^\rho$.
\end{enumerate}
It is now clear that (2) implies (1) by the fact that $\rho$ is a positive alliance if and only if  $(\{x\}^\rho)_\rho \subseteq \{x\}^\rho$ 
for all $x \in U$.
Next we give a counter example showing that (1) does not imply (2).

\begin{figure}[ht]
\centering
\includegraphics[width=25mm]{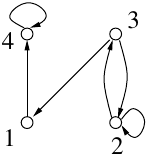}
\caption{\label{Fig:relation} \small
Positive alliance relation $\rho$. 
Elements of $U = \{1,2,3,4\}$ are represented with circles, and if an 
element $x$ is $\rho$-related to an element $y$, there is an arrow from the 
circle representing $x$ to the circle representing $y$.}
\end{figure}

\noindent%
Let the relation $\rho$ be given in Figure~\ref{Fig:relation}. 
The relation $\rho$ is a positive alliance, because:
\begin{itemize}
\item $1 \, \rho^c 1$, and there is $4$ such $1 \, \rho \, 4$ and $4 \, \rho^c \, 1$;
\item $1 \, \rho^c 2$, and there is $4$ such $1 \, \rho \, 4$ and $4 \, \rho^c \, 2$;
\item $1 \, \rho^c 3$, and there is $4$ such $1 \, \rho \, 4$ and $4 \, \rho^c \, 3$;
\item $2 \, \rho^c 1$, and there is $2$ such $2 \, \rho \, 2$ and $2 \, \rho^c \, 1$;
\item $2 \, \rho^c 4$, and there is $2$ such $2 \, \rho \, 2$ and $2 \, \rho^c \, 4$;
\item $3 \, \rho^c 3$, and there is $1$ such $3 \, \rho \, 1$ and $1 \, \rho^c \, 3$;
\item $3 \, \rho^c 4$, and there is $2$ such $3 \, \rho \, 2$ and $2 \, \rho^c \, 4$;
\item $4 \, \rho^c 1$, and there is $4$ such $4 \, \rho \, 4$ and $4 \, \rho^c \, 1$;
\item $4 \, \rho^c 2$, and there is $4$ such $4 \, \rho \, 4$ and $4 \, \rho^c \, 2$;
\item $4 \, \rho^c 3$, and there is $4$ such $4 \, \rho \, 4$ and $4 \, \rho^c \, 3$.  
\end{itemize}
We have that
\[ \rho(1) = \{4\}, \quad \rho(2) = \{2,3\}, \quad \rho(3) = \{1,2\}, \quad \rho(4) = \{4\}. \]
Let us consider the set $\{3,4\}$. Now $\{3,4\}^\rho = \{1,2,4\}$ and $(\{3,4\}^\rho)_\rho = \{1,2,4\}_\rho = \{1,3,4\}$. This means that
\[ (\{3,4\}^\rho)_\rho \nsubseteq \{3,4\}^\rho. \]
This obviously implies that there is a positive alliance $\rho$ such that
\[ (\{3,4\}^\rho)^c \nsubseteq ( (\{3,4\}^\rho)_\rho )^c =  ((\{3,4\}^\rho)^c)^\rho ,\]
contradicting Theorem~2(5) in \cite{Ma2015}.
\end{example}

For $L$-relations, we can present the following generalized definition.

\begin{definition}
An $L$-relation $R$ is said to be a \emph{positive alliance} if
\[ R(x,y)' \leq {\bigvee}_z (R(x,z) \wedge R(z,y)' ). \]
\end{definition}
We can now write the following lemma characterizing positive alliance relations in terms of the approximations of the identity functions.
But as in the case of (crisp) binary relations, we obviously
cannot present a full correspondence result.

\begin{lemma} \label{lem:alliance}
An $L$-relation $R$ is a positive alliance if and only if for all $x \in U$, $\U (I_x) \geq \L \U (I_x)$.
\end{lemma}

\begin{proof}
Let $x,y \in I$. We have
\[ \L \U (I_y)(x)  = {\bigwedge}_z (R(x,z)' \vee \U (I_y(z)) ) = {\bigwedge}_z \big (R(x,z)' \vee R(z,y) \big ) .\]
Thus,
\[ (\L \U (I_y)(x) )'  = \Big ( {\bigwedge}_z (R(x,z)' \vee R(z,y)) \Big )' = {\bigvee}_z \big (R(x,z) \wedge R(z,y)' \big ) .\]
Now, $\U (I_x) \geq \L \U (I_x)$ for all $x \in U$ if and only if $\U (I_y)(x) \geq \L \U (I_y)(x)$ for all $x,y \in U$ if and only if
$(\U (I_y)(x))' \leq (\L \U ((I_y)(x))'$ for all $x,y \in U$, and the equivalence follows from this. 
\end{proof}

\subsection{Serial relations}
\label{subsec:serial}

A binary relation $\rho$ on $U$ is serial if for all $x \in U$,
there exists $y \in U$ such that $x \, R \, y$. It is known that
being serial is equivalent to the fact that $X_\rho \subseteq X^\rho$
for all $X \subseteq U$. 

Our aim of this section is to find a definition for serial $L$-relations 
on $U$ such that a relation is serial if and only if
\begin{equation} \label{eq:inclusion}
(\forall A \in \mathcal{F}_L(U)) \, \L A \leq \U A.    
\end{equation}
In this subsection, we consider three definitions.
Let us begin with the following one.

\begin{definition} \label{def:serial1}
An $L$-relation is \emph{serial} if 
$(\forall x\in U)(\exists y \in U) \, R(x,y) = 1$.
\end{definition}

\begin{lemma} \label{lem:serial_implication}
If an $L$-relation is serial in the sense of Definition~\ref{def:serial1},
then \eqref{eq:inclusion} holds.
\end{lemma}

\begin{proof} Let $x \in U$. There exists $z \in U$ such that $R(x,z) = 1$. Now
  \[ \L(A)(x) = {\bigwedge}_y \big ( R(x,y)' \vee A(y) \big ) \leq R(x,z)' \vee A(z) = 1' \vee A(z) = 0 \vee A(z) = A(z) .\]
On the other hand,
\[ \U(A)(x) = {\bigvee}_y \big ( R(x,y) \wedge A(y) \big ) \geq R(x,z) \wedge A(z) = 1 \wedge A(z) = A(z) .\]
We have that $\L(A)(x) \leq A(z) \leq \U(A)(x)$ and the claim is proved.
\end{proof}

The converse implication does not hold, as shown in our following example.

\begin{example} \label{exa:serial_not_one}
Let us consider the 4-element lattice $L := 0 < a,b < 1$,
where $a$ and $b$ are incomparable. If we define $0'=1$, $a'=b$, $b'=a$, and $1' = 0$, then $\mathbf{L}$ forms a complete
De~Morgan Heyting algebra. 

Note that $L$ is now a Boolean algebra. This means that for all elements
$x$ and $y$ of $L$, the join $x' \vee y$ coincides with
$x \Rightarrow y$, where $\Rightarrow$ is the Heyting implication of $L$.  Hence, for any $L$-relation $R$, 
$\L(A) = {\bigwedge}_y (R(x,y) \Rightarrow A(y))$. Therefore, 
we can apply Proposition~3 of \cite{radzikowska2004fuzzy}.

Let $U = \{x,y\}$ and define $R(x,x) = R(y,y) = a$ and 
$R(x,y) = R(y,x) = b$. Thus, $R(x,x) \vee R(x,y) = a \vee b = 1$
and $R(y,y) \vee R(y,x) = 1$. Now  \cite[Proposition~3]{radzikowska2004fuzzy}
implies that $\L A \leq \U A$ for all $A \in \mathcal{F}_A$. 
But now $R$ is not serial in the sense of Definition~\ref{def:serial1}.
\end{example}

We have now showed that seriality of Definition~\ref{def:serial1} is not equivalent to condition \eqref{eq:inclusion}. Let us try a different 
approach, expressed in the following definition.

\begin{definition} \label{def:serial2}
 An $L$-relation is \emph{serial} if
\[
 (\forall x,y \in U) \, R(x,y)' \leq {\bigvee}_{z \neq y} R(x,z).
\]   
\end{definition}
Obviously, also this definition can be seen as a fuzzy generalization
of serial binary relations. Indeed, let $\rho$ be a serial binary relation 
on $U$. If $x \, \rho \, y$ does not hold some $y$, then there exists 
$z \neq y$ such that $x \, R \, z$ holds. 

\begin{proposition}
An $L$-fuzzy relation is serial in the sense of 
Definition~\ref{def:serial2} if and only if   
\begin{equation} \label{eq:inclusion_for_x}
(\forall x \in U) \, \L(I_x) \leq \U(I_x).  
\end{equation}
\end{proposition}

\begin{proof} For all $x,y \in U$, 
\[ \L(I_y)(x) = {\bigwedge}_z  ( R(x,z)' \vee I_y(z) ). \] 
If $z \neq y$, $R(x,z)' \vee I_y(z) = R(x,y)'$ and
if $z = y$, then $R(x,z)' \vee I_y(z) = 1$. Thus,
\[
\L(I_y)(x) = \big ( {\bigwedge}_{z \neq y} R(x,z)' \big ) \wedge 1 = 
\big ( {\bigwedge}_{z \neq y} R(x,z)' \big ) =
\big ( {\bigvee}_{z \neq y} R(x,z) \big )'. 
\]
On the other hand, by \eqref{eq:Upper_I_y}, $\U(I_y)(x) = R(x,y)$
for all $x,y \in U$. Therefore, \eqref{eq:inclusion_for_x}
is equivalent to 
 
\[ (\forall x,y \in U) \, \Big ( {\bigvee}_{z \neq y} R(x,z) \Big )' \leq R(x,y). \]
By applying the operation $'$ to both sides of the inequality, we
obtain the condition of Definition~\ref{def:serial2}.
\end{proof}

Unfortunately, \eqref{eq:inclusion_for_x} is not equivalent to
\eqref{eq:inclusion}, as can be seen in the following example.

\begin{example}
Let $L := 0 < a,b < 1$, where $a$ and $b$ are incomparable. 
Define $0' = 1$, $a' = a$, $b' = b$, and $1' = 0$. We set $U = \{x,y\}$
and define the $L$-relation $R$ on $U$ by $R(x,x) = R(x,y) = a$ and $R(y,x) = R(y,y) = b$.

Let us consider the $L$-set $\mathbf{0}$ defined by $x \mapsto 0$. We have that
\begin{gather*}
  \L (\mathbf{0})(x) = {\bigwedge}_w (R(x,w)' \vee \mathbf{0}(w)) = 
  (R(x,x)' \vee \mathbf{0}(x)) \wedge (R(x,y)' \vee \mathbf{0}(y)) = \\
  R(x,x)' \wedge R(x,y)' = a' \vee a' = a \vee a = a
\end{gather*}
and
\begin{gather*}
 \U(\mathbf{0})(x) = {\bigvee}_w (R(x,w) \wedge \mathbf{0}(w)) = 0.  
\end{gather*}
This means that $\L (\mathbf{0}) \nleq  \U(\mathbf{0})$. On the other hand,
\begin{align*}
  \U(I_x)(x) &= R(x,x) = a, \\
  \U(I_x)(y) &= R(y,x) = b, \\
  \U(I_y)(x) &= R(x,y) = a, \\
  \U(I_y)(y) &= R(y,y) = b
\end{align*}
and
\begin{align*}
  \L(I_x)(x) &= {\bigwedge}_z ( R(x,z)' \vee I_x(z)) = (R(x,x)' \vee I_x(x)) \wedge (R(x,y)' \vee I_x(y)) \\
             &= 1 \wedge R(x,y)' = R(x,y)' = a' = a, \\
  \L(I_x)(y) &= {\bigwedge}_z ( R(y,z)' \vee I_x(z)) = (R(y,x)' \vee I_x(x)) \wedge (R(y,y)' \vee I_x(y)) \\
             &= 1 \wedge R(y,y)' = R(y,y)' = b' = b, \\
  \L(I_y)(x) &= {\bigwedge}_z (R(x,z)' \vee I_y(z)) = (R(x,x)' \vee I_y(x)) \wedge (R(x,y)' \vee I_y(y)) \\
             &= R(x,x)' \wedge 1 = R(x,x)' = a' = a, \\
  \L(I_y)(y) &= {\bigwedge}_z (R(y,z)' \vee I_y(z)) = (R(y,x)' \vee I_y(x)) \wedge (R(y,y)' \vee I_y(y)) \\
             &= R(y,x)' \wedge 1 = R(y,x)' = b' = b. 
\end{align*}
We have shown that $\L(I_x) \leq \U(I_x)$ for all $x \in U$. This means 
that \eqref{eq:inclusion_for_x} is not equivalent to \eqref{eq:inclusion}.
\end{example}

In case of binary relations, it is also true that a relation
$\rho$ is serial if and only if $U^\rho = U$. We end this section
by showing that we can have a simple correspondence generalizing this.

\begin{definition} \label{def:serial3}
An $L$-relation $R$ is \emph{serial} if 
$(\forall x \in U) \, \bigvee_y R(x,y) = 1$. 
\end{definition}

Now we can write to following characterization.

\begin{lemma} 
An $L$-relation $R$ is serial if and only if $\U(\mathbf{1}) = \mathbf{1}$.
\end{lemma}

\begin{proof} Because
  \[ \U(\mathbf{1})(x) = {\bigvee}_y (R(x,y) \wedge \mathbf{1}(y)) = {\bigvee}_y R(x,y) , \]
the claim is obvious.  
\end{proof}

It is clear that the seriality of Definition~\ref{def:serial1} implies 
the seriality of Definition~\ref{def:serial3}.
The converse is not true. For instance, the $L$-relation $R$ of 
Example~\ref{exa:serial_not_one} is serial in the sense of
Definition~\ref{def:serial3}, but 
$R(x,y) \neq 1$ for all $x,y \in U$.

We also proved in Lemma~\ref{lem:serial_implication} that
the seriality of Definition~\ref{def:serial1} implies 
\eqref{eq:inclusion}. Our final example of this subsection
shows that seriality of Definition~\ref{def:serial3} does
not imply \eqref{eq:inclusion}.

\begin{example} \label{Exa:serial1}
Let $L$ be the 4-element lattice $0 < a,b < 1$, where $a$ and $b$ 
are incomparable. If we define $0'=1$, $a'=a$, $b'=b$, and $1' = 0$, 
then we have a complete De~Morgan Heyting algebra. 
Let $U = \{x,y\}$ and define 
$R(x,x) = R(y,y) = a$ and $R(x,y) = R(y,x) = b$. 
It is now clear that $R$ is serial in the sense of 
Definition~\ref{def:serial3}, because $R(x,x) \vee R(x,y) = a \vee b = 1$ and $R(y,x) \vee R(y,y) = b \vee a = 1$.

Let $U = \{a,b\}$ and define an $L$-set $A \colon U \to L$ by $A(x) = b$ and $A(y) = a$. Now
\begin{gather*} \L(A) (x) = (R(x,x)' \vee A(x)) \wedge (R(x,y)' \vee A(y)) = (a' \vee b) \wedge (b' \vee a) \\
  = (a \vee b) \wedge (b \vee a) = 1 \wedge 1 = 1.
\end{gather*}
and
\begin{gather*}
  \U(A) (x) = (R(x,x) \wedge A(x)) \vee (R(x,y) \wedge A(y)) = (a \wedge b) \vee (b \wedge a) = 0 \vee 0 = 0.
\end{gather*}
Therefore, $\L(A) \nleq \U(A)$.
\end{example}

\section{From {\em L}-approximations to {\em L}-relations}
\label{Sec:Axiomatization}

In the previous section, we defined the $L$-approximations $\L A$ and $\U A$ for any $A \in \mathcal{F}_L(U)$ in terms of an $L$-relation $R$ on $U$. 
In this section, we consider a converse problem, that is, whether we can 
define an $L$-relation of certain type
for a dual pair of $L$-fuzzy operators. As we already noted in 
Section~\ref{subsec:demorgan_heyting}, for every $L$-relation $R$ on $U$, 
we have 
\[ \U (I_y)(x) = R(x,y), \]
for all $x, y\in U$ and  $A \in \mathcal{F}_L(U)$. This provides a `rule' for defining relations when upper approximations are known.
We call any map on $\mathcal{F}_L(U)$ as an \emph{$L$-fuzzy operator} on $U$. 

\begin{lemma} \label{lem:equations}
Let $\UU$ be an $L$-fuzzy operator. For any $a \in L$, $A \in \mathcal{F}_L(U)$ and $\{A_i\}_{i\in I} \subseteq \mathcal{F}_L(U)$ the following are equivalent:
\begin{enumerate}[label = {\rm (\arabic*)}]
\item $\UU(A) = \bigvee_x \big (\overline{A(x)} \wedge \UU (I_x) \big )$;
\item $\UU (\overline{a} \wedge \bigvee_i A_i) = \overline{a} \wedge \bigvee_i \UU(A_i)$.
\end{enumerate}
\end{lemma}

\begin{proof}
(1)$\Rightarrow$(2): Suppose that (1) holds. Then,
\begin{align*}
\UU (\overline{a} \wedge A) &= {\bigvee}_x \big (\overline{(\overline{a} \wedge A)(x)} \wedge \UU (I_x) \big ) 
= {\bigvee}_x \big (\overline{a \wedge A(x)} \wedge \UU (I_x) \big ) \\
&= {\bigvee}_x \big ( \overline{a} \wedge \overline{A(x)}) \wedge \UU(I_x) \big )  
= \overline{a} \wedge {\bigvee}_x \big (\overline{A(x)}) \wedge \UU(I_x) \big )  \\
&= \overline{a} \wedge \UU(A)
\end{align*}
and 
\begin{align*}
\UU \big ({\bigvee}_i A_i \big ) &= {\bigvee}_x \big (\overline{ \big ({\bigvee}_i A_i \big )(x)} \wedge \UU (I_x) \big ) 
= {\bigvee}_x \big ({\bigvee}_i \overline{ (A_i)(x)} \wedge \UU (I_x) \big ) \\
&= {\bigvee}_x {\bigvee}_i \left (\overline{A_i(x)} \wedge \UU (I_x) \right)  
= {\bigvee}_i {\bigvee}_x \left (\overline{A_i(x)} \wedge \UU (I_x) \right)  \\
&= {\bigvee}_i \, \UU (A_i).  
\end{align*}
Combining these two, we obtain that
\[ \UU \left ( \overline{a} \wedge {\bigvee}_i A_i \right ) = \overline{a} \wedge \UU \left ({\bigvee}_i A \right ) = \overline {a} \wedge {\bigvee}_i \, \UU (A_i).\]

\noindent%
(2)$\Rightarrow$(1): Suppose that (2) holds. Since $A = \bigvee_x (\overline{A(x)} \wedge I_x)$, we have that
\begin{align*}
\UU (A) & = \UU \left ( {\bigvee}_x \left ( \overline{A(x)} \wedge I_x \right ) \right ) 
 = \UU \left ( \overline{1} \wedge {\bigvee}_x \left (\overline{A(x)} \wedge I_x \right ) \right ) \\  
& = \overline{1} \wedge {\bigvee}_x \, \UU \left (\overline{A(x)} \wedge I_x \right ) 
 = {\bigvee}_x \UU \left (\overline{A(x)} \wedge I_x \right ) \\
&= {\bigvee}_x \left ( \overline{A(x)} \wedge \UU (I_x) \right ).
\qedhere
\end{align*}
\end{proof}

\begin{theorem} \label{thm:main1}
Let $\UU$ be an $L$-fuzzy operator on $U$ and let each $\SS_j$, $1 \leq j \leq n$, be an $L$-fuzzy operator on $U$ such that $\UU \leq S_j$.
Then there exists a unique $L$-fuzzy relation $R$ on $U$ such that $\U(A) = \UU(A)$ for all $A \in \mathcal{F}_L(U)$
if and only if
\begin{equation}\label{eq:multiple} 
\UU(\overline{a} \wedge {\bigvee}_i A_i) = \overline{a} \wedge {\bigvee}_i \left (\UU(A_i) \wedge \SS_1(A_i) \wedge \cdots \wedge \SS_n(A_i) \right )
\end{equation}
for $a \in L$ and $\{A_i\}_{i \in I} \subseteq \mathcal{F}_L(U)$. 
\end{theorem}

\begin{proof} 
$(\Rightarrow)$ Since $\UU \leq \SS_j$, for all $j \leq n$, we have $\UU \leq \SS_1 \wedge \cdots \wedge \SS_n$ and hence
\[ \UU = \UU \wedge \SS_1 \wedge \cdots \wedge \SS_n .\]
Suppose that there is an $L$-relation $R$ such that $\UU(A) = \U(A)$ for all $A \in \mathcal{F}_L(U)$. Then,    
\begin{align*}
  \UU \big (\overline{a}\wedge {\bigvee}_i A_i \big ) & = \U \big (\overline{a} \wedge {\bigvee}_i A_i \big )
  = \Big (\overline{a} \wedge \big ({\bigvee}_i \U (A_i) \big ) \Big ) = \overline{a}\wedge {\bigvee}_i \U (A_i) \\
& =\overline{a}\wedge {\bigvee}_i \UU (A_i) = \overline{a} \wedge {\bigvee}_i (\UU (A_i) \wedge \SS_1 (A_i) \wedge \cdots \wedge \SS_n(A_i)).
\end{align*}

\noindent%
$(\Leftarrow)$ Suppose that \eqref{eq:multiple} holds for $a \in L$ and $\{A_i\}_{i \in I} \subseteq \mathcal{F}_L(U)$. This gives
\[ \UU \big ( {\bigvee}_i \, A_i \big ) = \UU \big (\overline{1} \wedge {\bigvee}_i \, A_i \big ) =
\overline{1} \wedge {\bigvee}_i \, \UU(A_i) = {\bigvee}_i \, \UU(A_i) .\]
and $\UU (\overline{a} \wedge A) = \overline{a} \wedge \UU(A)$ for all $a \in L$ and $A \in \mathcal{F}_L(U)$, which are essential properties
of upper approximation operators. 

Let us define an $L$-relation $R$ on $U$ by setting for all $x,y \in U$, 
\[ R(x,y) = \UU(I_y)(x) .\]
Then, $\U (I_y)(x) = R(x,y) = \UU(I_y)(x)$ for all $x,y \in U$, and hence $\U(I_x) = \UU(I_x)$ for all $x\in U$. For $A \in \mathcal{F}_L(U)$, 
\begin{align*}
\UU(A) &= \UU \big ({\bigvee}_x \Big ( \overline{A(x)} \wedge I_x \big ) \Big )
= {\bigvee}_x \UU \big (\overline{A(x)} \wedge I_x \big ) \\
&= {\bigvee}_x  (\overline{A(x)} \wedge \UU(I_x)  ) 
= {\bigvee}_x (\overline{A(x)} \wedge \U (I_x) ) \\
&= {\bigvee}_x \U \big (\overline{A(x)} \wedge I_x \big )
= \U \Big ({\bigvee}_x \big (\overline{A(x)} \wedge I_x \big) \Big) \\
&= \U A.
\end{align*}
It is also clear that the induced relation $R$ is unique, since if there is an $L$-relation $Q$ such that $\U_Q (A) = \UU (A)$ for all $A \in \mathcal{F}_L(U)$, we have
$R(x,y) = \UU(I_y)(x) = \U_Q (I_y)(x) = Q(x,y)$ for all $x,y\in U$ and thus $R=Q$. 
\end{proof}

Let $\UU$ and $\LL$ be $L$-fuzzy operators. 
We say that $\UU$ and $\LL$ are \emph{dual} if $\UU(A') = (\LL(A))'$ for all $A \in \mathcal{F}_L(U)$. \label{page:dual}
Note that this is equivalent to $\LL(A') = (\UU(A))'$ for all $A \in \mathcal{F}_L(U)$. Notice also that we used the notion of dual operations already in Section~\ref{Sec:Preliminaries}, even it was not defined there. The dual operators define each other uniquely, that is,
\[ \UU(A) = (\LL(A'))' \quad \text{and} \quad \LL(A) = (\UU(A'))' , \]
for all $A \in \mathcal{F}_L(U)$. 

\begin{remark}
Suppose that the $L$-fuzzy operators $\UU$ and $\LL$ on $U$ are dual. Theorem~\ref{thm:main1} can be equivalently expressed in terms of $\LL$. Indeed, let
for any $1 \leq j \leq n$, each $\SS_j$ be an $L$-fuzzy operator on $U$ such that $S_j(A) \leq \LL(A)$ for any $A \in \mathcal{F}_L(A)$. Then there exists a unique
$L$-fuzzy relation $R$ on $U$ such that $\L A = \LL (A)$ and 
$\U A = \UU(A)$ for all $A \in \mathcal{F}_L(U)$ if and only if
\[
\LL \big ( \overline{a} \vee {\bigwedge}_i A_i \big ) = \overline{a} \vee {\bigwedge}_i \left( \LL(A_i) \vee \SS_1(A_i) \vee \cdots \vee \SS_n (A_i) \right )  
\]
for all $a\in L$ and  $A \in \mathcal{F}_L(U)$. 
\end{remark}

As a corollary of \ref{thm:main1}, we can get easily the following results, which were presented in \cite{Pang2019} (see Theorems~4.2--4.6). Note that case (5) is a new result.

\begin{corollary}
Let $\UU$ and $\LL$ be dual $L$-fuzzy operators on $U$.
\begin{enumerate}[label = {\rm (\arabic*)}]
\item There exists a unique $L$-relation $R$ on $U$ such that $\UU$ and $\LL$ coincide with the upper and lower approximation operators of $R$, respectively, if and only if
\[
\UU (\overline{a} \wedge {\bigvee}_i A_i) = \overline{a} \wedge {\bigvee}_i \UU(A_i) 
\]
for all $a \in L$ and $\{A_i\}_{i \in I} \subseteq \mathcal{F}_L(U)$. 

\item There exists a unique mediate $L$-relation $R$ on $U$ such that $\UU$ and $\LL$ coincide with the upper and lower approximation operators of $R$, respectively, if and only if
\[
\UU(\overline{a} \wedge {\bigvee}_i A_i) = \overline{a} \wedge {\bigvee}_i   (\UU(A_i) \wedge \UU (\UU(A_i))  ) 
\]
for all $a \in L$ and $\{A_i\}_{i \in I} \subseteq \mathcal{F}_L(U)$. 

\item There exists a unique Euclidean $L$-relation $R$ on $U$ such that $\UU$ and $\LL$ coincide with the upper and lower approximation operators of $R$, respectively, if and only if
\[
\UU(\overline{a} \wedge {\bigvee}_i A_i) = \overline{a} \wedge {\bigvee}_i  (\UU(A_i) \wedge \LL (\UU(A_i))   ) 
\]
for all $a \in L$ and $\{A_i\}_{i \in I} \subseteq \mathcal{F}_L(U)$.

\item There exists a unique adjoint $L$-relation $R$ on $U$ such that $\UU$ and $\LL$ coincide with the upper and lower approximation operators of $R$, respectively, if and only if
\[
\UU(\overline{a} \wedge {\bigvee}_i A_i) = \overline{a} \wedge {\bigvee}_i   (\UU(A_i) \wedge \UU (\LL(A_i))   ) 
\]
for all $a \in L$ and $\{A_i\}_{i \in I} \subseteq \mathcal{F}_L(U)$.

\item There exists a unique functional $L$-relation $R$ on $U$ such that $\UU$ and $\LL$ coincide with the upper and lower approximation operators of $R$, respectively, if and only if
\[
\UU(\overline{a} \wedge {\bigvee}_i A_i) = \overline{a} \wedge {\bigvee}_i  (\UU(A_i) \wedge \LL(A_i)  ) 
\]
for all $a \in L$ and $\{A_i\}_{i \in I} \subseteq \mathcal{F}_L(U)$. 
\end{enumerate}
\end{corollary}

\begin{proof} (1) This follows directly from Theorem~\ref{thm:main1} by selecting $n = 1$ and $\SS_1 = \U$.

(2) By Proposition~\ref{prop:mediate}, $R$ is mediate if and only if $\U A \leq \U \U A$ for any $A \in \mathcal{F}_L(U)$ from which the result follows by
Theorem~\ref{thm:main1} by setting $n = 1$ and $\SS_1 = \U \U$. 

(3) The relation $R$ is Euclidean if and only if for all $A \in \mathcal{F}_L(U)$, $\U A \leq \L \U A$ by Proposition~\ref{prop:Eucl}. If we set $n = 1$ and
$\SS_1 = \L \U$, the claim follows from Theorem~\ref{thm:main1}.

(4) By Proposition~\ref{prop:adjoint}, $R$ is adjoint if and only if for all $A \in \mathcal{F}_L(U)$, $\U A \leq \U\L A$ and the claim follows from this by
Theorem~\ref{thm:main1}.

(5) The relation $R$ is functional if and only if $\U A \leq \L A$ for any $L$-set $A$ on $U$ by Proposition~\ref{prop:func}. The result is now clear by this.
\end{proof} 

Using Theorem~\ref{thm:main1} we can also give conditions for compositions of properties (cf.~Theorems 4.8, 4.9, 4.10, and 4.12 in  \cite{Pang2019}). The proof
of the following proposition is clear.

\begin{corollary}
Let $\UU$ and $\LL$ be dual $L$-fuzzy operator on $U$.
There exists a unique mediate, Euclidean and adjoint $L$-relation $R$ on $U$ such that $\UU$ and $\LL $ coincide with the upper and lower approximation operators of $R$,
respectively, if and only if
\[
\UU(\overline{a} \wedge {\bigvee}_i A_i) = \overline{a} \wedge {\bigvee}_i (\UU(A_i) \wedge \UU (\UU(A_i)) \wedge \LL(\UU(A_i)) \wedge \UU(\LL(A_i))) 
\]
for all $a \in L$ and $\{A_i\}_{i \in I} \subseteq \mathcal{F}_L(U)$. 
\end{corollary}

In  \cite{Pang2019}, the authors presented the following problem: ``Using single axioms to characterize $L$-fuzzy rough approximation operators corresponding to compositions
of serial, reflexive, symmetric, transitive, mediate, Euclidean and adjoint L-fuzzy relations.'' Next, we solve this problem so that only the parts concerning serial and symmetric relations remains open. 

Our following theorem is closely related to Theorem~\ref{thm:main1}.

\begin{theorem} \label{thm:main2}
Let $\UU$ be an $L$-fuzzy operator on $U$ and let each $\TT_k$, $1 \leq k \leq m$, be an $L$-fuzzy operator on $U$ such that
$\UU \geq \TT_k$. Then there exists a unique $L$-fuzzy relation $R$ on $U$ such that $\U A = \UU(A)$ for all $A \in \mathcal{F}_L(U)$
if and only if
\begin{equation}\label{eq:multiple2} 
\UU(\overline{a} \wedge {\bigvee}_i A_i) = \overline{a} \wedge {\bigvee}_i (\UU(A_i) \vee \TT_1(A_i) \vee \cdots \vee \TT_m(A_i) )
\end{equation}
for $a \in L$ and $\{A_i\}_{i \in I} \subseteq \mathcal{F}_L(U)$. 
\end{theorem}

\begin{proof} 
$(\Rightarrow)$ Since $\UU \geq \TT_k$, for all $k \leq m$, we have $\UU \leq \TT_1 \vee \cdots \vee \TT_m$ and hence
\[ \UU = \UU \vee \TT_1 \vee \cdots \vee \TT_m .\]
If there is an $L$-relation $R$ such that $\U A = \UU A$ for all $A \in \mathcal{F}_L(U)$, then    
\begin{align*}
  \UU \big (\overline{a} \wedge {\bigvee}_i A_i \big ) & = 
  \U  \big (\overline{a} \wedge {\bigvee}_i A_i \big)
  = \big (\overline{a} \wedge \U \big ({\bigvee}_i A_i \big) \big ) = \overline{a}\wedge {\bigvee}_i \U (A_i) \\
& =\overline{a} \wedge {\bigvee}_i \UU (A_i) = 
\overline{a} \wedge {\bigvee}_i (\UU(A_i) \vee \TT_1 (A_i) \vee \cdots \vee \TT_m(A_i) ).
\end{align*}

\noindent%
The direction $(\Leftarrow)$ is proved as in Theorem~\ref{thm:main1}. 
\end{proof}

Now we can write the following corollary of Theorem~\ref{thm:main2}. 

\begin{corollary}
Let $\UU$ and $\LL$ be dual $L$-fuzzy operators on $U$.
\begin{enumerate}[label = {\rm (\arabic*)}]

\item There exists a unique reflexive $L$-relation $R$ on $U$ such that $\UU$ and $\LL$ coincide with the upper and lower approximation operators of $R$, respectively, if and only if
\[
\UU(\overline{a} \wedge {\bigvee}_i A_i) = \overline{a} \wedge {\bigvee}_i \left (\UU(A_i) \vee  A_i) \right ) 
\]
for all $a \in L$ and $\{A_i\}_{i \in I} \subseteq \mathcal{F}_L(U)$. 

\item There exists a unique transitive $L$-relation $R$ on $U$ such that $\U$ and $\L$ coincide with the upper and lower approximation operators of $R$, respectively, if and only if
\[
\UU(\overline{a} \wedge {\bigvee}_i A_i) = \overline{a} \wedge {\bigvee}_i (\UU(A_i) \vee \UU (\UU (A_i)) ) 
\]
for all $a \in L$ and $\{A_i\}_{i \in I} \subseteq \mathcal{F}_L(U)$. 
\end{enumerate}
\end{corollary}

\begin{proof} (1) By Proposition~\ref{prop:reflexive}, $R$ is reflexive if and only if $A \leq \U A$ for any $A \in \mathcal{F}_L(U)$ from which the result follows by
Theorem~\ref{thm:main2} by setting $n = 1$ and $\TT_1(A) = A$. 

(2) By Proposition~\ref{prop:transitive}, $R$ is transitive if and only if $\U \U A \leq \U A$ for any $A \in \mathcal{F}_L(U)$ from which the result follows by
Theorem~\ref{thm:main2} by setting $n = 1$ and $\TT_1 = \U \U$. 
\end{proof}

Let $\UU$ be an $L$-fuzzy operator on $U$. Assume that each $\SS_j$, $1 \leq j \leq n$, is an $L$-fuzzy operator on $U$ such that $\UU \leq \SS_j$ and
suppose each $\TT_k$, $1 \leq k \leq m$ is an $L$-fuzzy operator on $U$ such that $\UU \geq \TT_k$. Now we have that
\begin{equation} \label{eq:combined}
  \UU = (\UU \wedge \SS_1 \wedge \cdots \wedge \SS_n) \vee (\TT_1 \vee \cdots \vee \TT_m).
\end{equation}
We can now write the following theorem. Its proof is clear, because $(\Rightarrow)$ part follows from \eqref{eq:combined},
and $(\Leftarrow)$ can be proved as in Theorem \ref{thm:main1}.

\begin{theorem} \label{thm:main3}
Let $\UU$ be an $L$-fuzzy operator on $U$ and let each $\SS_j$, $1 \leq j \leq n$ and $\TT_k$, $1 \leq k \leq m$, be $L$-fuzzy operators on $U$ such that
$\UU \leq \SS_j$ and $\UU \geq \TT_k$. Then there exists a unique $L$-fuzzy relation $R$ on $U$ such that $\U A = \UU(A)$ for all $A \in \mathcal{F}_L(U)$
if and only if
\[
\UU(\overline{a} \wedge {\bigvee}_i A_i) = \overline{a} \wedge {\bigvee}_i ( (\UU(A_i) \wedge \SS_1(A_i) \wedge \cdots \wedge \SS_n(A_i))
\vee (\TT_1(A_i) \vee \cdots \vee \TT_m(A_i)) )
\]
for $a \in L$ and $\{A_i\}_{i \in I} \subseteq \mathcal{F}_L(U)$. \qed 
\end{theorem}

We end this work by the following example of a single condition for a combination of certain relation types.

\begin{corollary}
Let $\UU$ and $\LL$ be dual $L$-fuzzy operators on $U$. There exists a reflexive, transitive, mediate, Euclidean and adjoint $L$-relation $R$ on $U$ such that $\UU$ and $\LL$
coincide with the upper and lower approximation operators of $R$,
respectively, if and only if
\[
 \UU(\overline{a} \wedge {\bigvee}_i A_i) 
 = \overline{a} \wedge {\bigvee}_i 
\big( 
(\UU(A_i) \wedge \UU (\UU(A_i)) \wedge \LL(\UU(A_i)) \wedge \UU(\LL(A_i)) )
\vee ( A_i \vee \UU(\UU(A_i)))
\big )
\]
for all $a \in L$ and $\{A_i\}_{i \in I} \subseteq \mathcal{F}_L(U)$. 
\end{corollary}

\section*{Some concluding remarks} 
In this work, we were able to solve the open problem presented in \cite{Pang2019} so that only the parts concerning serial and symmetric relations remain open. 

One obvious reason for this is that we do not have
a condition for symmetry or seriality which is of the suitable form so
that the theorems given in Section~\ref{Sec:Axiomatization}
could be applied. Interestingly, seriality and symmetry are also such that there are some issues how to generalize them as fuzzy relations, at least
in case of De~Morgan Heyting algebras.

\section*{Acknowledgements}

We thank the anonymous referees for their valuable remarks on our manuscript.

\section*{Compliance with Ethical Standards}

\paragraph*{Funding} The authors received no specific funding for this study.

\paragraph*{Conflict of interest}
Jouni J\"arvinen declares that he has no conflict of interest. 
Michiro Kondo declares that he has no conflict of interest.

\paragraph*{Ethical approval} This article does not contain any studies with human participants or animals performed by any of the authors.

\bibliographystyle{plain}
\bibliography{literature}

\begin{thebibliography}{10}

\bibitem{Castano11}
Valeria Casta{\~n}o and Marcela Mu{\~n}oz~Santis.
\newblock Subalgebras of {H}eyting and {D}e~{M}organ {H}eyting {A}lgebras.
\newblock {\em Studia Logica}, 98:123--139, 2011.

\bibitem{Deer2015}
Lynn D{'}eer, Nele Verbiest, Chris Cornelis, and Llu\'{i}s Godo.
\newblock A comprehensive study of implicator–conjunctor-based and noise-tolerant fuzzy rough sets: Definitions, properties and robustness analysis.
\newblock {\em Fuzzy Sets and Systems}, 275:1--38, 2015.

\bibitem{DuboisPrade1990}
Didier Dubois and Henri Prade.
\newblock Rough fuzzy sets and fuzzy rough sets.
\newblock {\em International Journal of General Systems}, 17:191--209, 1990.

\bibitem{Goguen1967}
J.A Goguen.
\newblock {$L$}-fuzzy sets.
\newblock {\em Journal of Mathematical Analysis and Applications}, 18:145--174, 1967.

\bibitem{Jarvinen2005}
Jouni J{\"a}rvinen.
\newblock Properties of rough approximations.
\newblock {\em J. Adv. Comput. Intell. Intell. Inform.}, 9:502--505, 2005.

\bibitem{Jin23}
Qiu Jin and Ling-Qiang.
\newblock Several {$L$}-fuzzy variable precision rough sets and their axiomatic characterizations.
\newblock {\em Soft Computing}, 27:16429--16448, 2023.

\bibitem{Liu08}
Guilong Liu.
\newblock Axiomatic systems for rough sets and fuzzy rough sets.
\newblock {\em International Journal of Approximate Reasoning}, 48(3):857--867, 2008.

\bibitem{Liu13}
Guilong Liu.
\newblock Using one axiom to characterize rough set and fuzzy rough set approximations.
\newblock {\em Information Sciences}, 223:285--296, 2013.

\bibitem{Ma2015}
Zhouming Ma, Jinjin Li, and Jusheng Mi.
\newblock Some minimal axiom sets of rough sets.
\newblock {\em Information Sciences}, 312:40 -- 54, 2015.

\bibitem{Orlowska1998}
Ewa Or{\l}owska.
\newblock Introduction: What you always wanted to know about rough sets.
\newblock In Ewa Or{\l}owska, editor, {\em Incomplete Information: Rough Set Analysis}, pages 1--20. Physica-Verlag, Heidelberg, 1998.

\bibitem{Pang2019}
Bin Pang, Ju-Sheng Mi, and Wei Yao.
\newblock {$L$}-fuzzy rough approximation operators via three new types of {$L$}-fuzzy relations.
\newblock {\em Soft Computing}, 23:11433--11446, 2019.

\bibitem{Pavelka1979}
Jan Pavelka.
\newblock On fuzzy logic {II}. {E}nriched residuated lattices and semantics of propositional calculi.
\newblock {\em Mathematical Logic Quarterly}, 25:119--134, 1979.

\bibitem{Pawlak82}
Zdzis{\l}aw Pawlak.
\newblock Rough sets.
\newblock {\em International Journal of Computer and Information Sciences}, 11:341--356, 1982.

\bibitem{radzikowska2004fuzzy}
Anna~Maria Radzikowska and Etienne~E Kerre.
\newblock Fuzzy rough sets based on residuated lattices.
\newblock {\em Transactions on Rough Sets II}, II:278--296, 2004.

\bibitem{Sun20}
Shoubin Sun, Lingqiang Li, Kai Hu, and A.~A. Ramadan.
\newblock {$L$}-fuzzy upper approximation operators associated with $l$-generalized fuzzy remote neighborhood systems of $l$-fuzzy points.
\newblock {\em AIMS Mathematics}, 5:5639--5653, 2020.

\bibitem{Sun23}
Yan Sun and Fu-Gui Shi.
\newblock Representations of {$L$}-fuzzy rough approximation operators.
\newblock {\em Information Sciences}, 645, 2023.
\newblock In print.

\bibitem{Wei21}
Xiaowei Wei, Bin Pang, and Ju-Sheng Mi.
\newblock Axiomatic characterizations of {$L$}-valued rough sets using a single axiom.
\newblock {\em Information Sciences}, 580:283--310, 2021.

\bibitem{WuZang2004}
Wei-Zhi Wu and Wen-Xiu Zhang.
\newblock Constructive and axiomatic approaches of fuzzy approximation operators.
\newblock {\em Information Sciences}, 159:233--254, 2004.

\bibitem{Xu23}
Yao-Liang Xu, Dan-Dan Zou, Ling-Qiang Li, and Bing-Xue Yao.
\newblock {$L$}-fuzzy covering rough sets based on complete co-residuated lattice.
\newblock {\em International Journal of Machine Learning and Cybernetics}, 14:2815--2829, 2023.

\bibitem{Yao96}
Y.Y. Yao and T.Y. Lin.
\newblock Generalization of rough sets using modal logics.
\newblock {\em Intelligent Automation \& Soft Computing}, 2:103--119, 1996.

\bibitem{Zadeh65}
L.A. Zadeh.
\newblock Fuzzy sets.
\newblock {\em Information and Control}, 8:338--353, 1965.

\bibitem{Zhao21}
FangFang Zhao and Fu-Gui Shi.
\newblock {$L$}-fuzzy generalized neighborhood system operator-based {$L$}-fuzzy approximation operators.
\newblock {\em International Journal of General Systems}, 50:458--484, 2021.

\bibitem{Zhu2007}
William Zhu.
\newblock Generalized rough sets based on relations.
\newblock {\em Information Sciences}, 177(22):4997 -- 5011, 2007.

\end{thebibliography}
\end{document}